\documentclass[12pt]{amsart}
\usepackage{amscd,amssymb}%,showkeys}
\usepackage[mathscr]{eucal}
\usepackage[english]{babel}
\input pstricks
\input pst-node

\usepackage{pstricks, pst-node}
\input xy
\usepackage[all]{xy}

\usepackage{amsmath}
\usepackage[latin1]{inputenc}%
\usepackage{a4}

\newlength{\itemlaenge}

\newtheoremstyle{mytheorem}% name
  {}%      Space above, empty = `usual value'
  {}%      Space below
  {\slshape}% Body font
  {}%         Indent amount (empty = no indent, \parindent = para indent)
  {\scshape}% Thm head font
  {.}%        Punctuation after head
  { }%     Space after thm head: " " = normal interword space;
  {}% Thm head spec

\newtheoremstyle{mydefinition}% name
  {}%      Space above, empty = `usual value'
  {}%      Space below
  {\upshape}% Body font
  {}%         Indent amount (empty = no indent, \parindent = para indent)
  {\scshape}% Thm head font
  {.}%        Punctuation after thm head
  { }%     Space after thm head: " " = normal interword space;
  {}% Thm head spec

\theoremstyle{mytheorem}
\newtheorem{lemma}{Lemma}[section]
\newtheorem{prop}[lemma]{Proposition}
\newtheorem*{prop*}{Proposition}
\newtheorem{prop_intro}{Proposition}

\newtheorem{thm_intro}[prop_intro]{Theorem}

\newtheorem*{thm*}{Theorem}
\theoremstyle{mydefinition}
\newtheorem{rem}[lemma]{Remark}
\newtheorem*{rem*}{Remark}

\newtheorem*{notation*}{Notation}

\newtheorem*{warning*}{Warning}

\newtheorem*{defi*}{Definition}

\numberwithin{equation}{section}

\newcommand{\bqn}{\begin{equation*}}
\newcommand{\eqn}{\end{equation*}}
\newcommand{\bq}{\begin{equation}}
\newcommand{\eq}{\end{equation}}
\newcommand{\ba}{\begin{aligned}}
\newcommand{\ea}{\end{aligned}}
\newcommand{\be}{\begin{enumerate}}
\newcommand{\ee}{\end{enumerate}}

\newcommand{\thismonth}{\ifcase\month % case 0 --- impossible!
  \or January\or February\or March\or April\or May\or June%
  \or July\or August\or September\or October\or November%
  \or December\fi}

\newcommand{\RR}{{\mathbb R}}

\newcommand{\Ff}{{\mathcal F}}

\def\h{{\rm H}}
\def\hb{{\rm H}_{\rm b}}

\def\cb{{\rm C}_{\rm b}}
\def\ch{{\rm C}}
\def\bdh{{\rm B}}

\def\one{\mathbf{1\kern-1.6mm 1}}

\def\id{{\it I\! d}}

\def\h2{{\operatorname{H_2}}}
\def\h1{{\operatorname{H_1}}}

\def\id{{\operatorname{Id}}}

\def\to{\rightarrow}

\def\hb{{\rm H}_{\rm b}}

\def\h{{\rm H}}

\renewcommand{\phi}{\varphi}

\def\No{N\raise4pt\hbox{\tiny o}\kern+.2em}
\def\no{n\raise4pt\hbox{\tiny o}\kern+.2em}

\begin{document}

\title[Isometric properties of  relative bounded cohomology]{Isometric properties of \\   relative bounded cohomology}

\author[M. Bucher et al.]{M. Bucher}
\address{Section de Math\'ematiques Universit\'e de Gen\`eve, 
2-4 rue du Li\`evre, Case postale 64, 1211 Gen\`eve 4, Suisse}
\email{Michelle.Bucher-Karlsson@unige.ch}

\author[]{M. Burger}
\address{Department Mathematik, ETH Zentrum, 
R\"amistrasse 101, CH-8092 Z\"urich, Switzerland}
\email{burger@math.ethz.ch}

\author[]{R. Frigerio}
\address{Dipartimento di Matematica, Universit\`a di Pisa, Largo B. Pontecorvo 5, 56127 Pisa, Italy}
\email{frigerio@dm.unipi.it}

\author[]{A. Iozzi}
\address{Department Mathematik, ETH Zentrum, 
R\"amistrasse 101, CH-8092 Z\"urich, Switzerland}
\email{iozzi@math.ethz.ch}

\author[]{C. Pagliantini}
\address{Dipartimento di Matematica, Universit\`a di Pisa, Largo B. Pontecorvo 5, 56127 Pisa, Italy}
\email{pagliantini@mail.dm.unipi.it}

\author[]{M. B. Pozzetti}%{Maria Beatrice Pozzetti}
\address{Department Mathematik, ETH Zentrum, 
R\"amistrasse 101, CH-8092 Z\"urich, Switzerland}
\email{beatrice.pozzetti@math.ethz.ch}

\thanks{Michelle Bucher was supported by Swiss National Science Foundation 
project PP00P2-128309/1, Alessandra Iozzi was partial supported by the 
Swiss National Science Foundation project 2000021-127016/2.  The first five
named authors thank the Institute Mittag-Leffler in Djursholm, Sweden, 
for their warm hospitality during  the preparation of this paper.}



\date{\today}

\begin{abstract} 
We show that the morphism induced by the inclusion of pairs $(X,\emptyset)\subset (X,Y)$ 
between the relative bounded cohomology of $(X,Y)$ and 
the bounded cohomology of $X$ is an {\it isometric} isomorphism in degree at least $2$ 
if the fundamental group of each connected component of $Y$ is amenable.
As an application, we provide a self-contained proof of Gromov's Equivalence
Theorem and a generalization of a result by Fujiwara and Manning
on the simplicial volume of generalized Dehn fillings.
\end{abstract}
\maketitle
%
%
%\tableofcontents
%
%
%%%%%%%%%%%%%%%%%%%%%%%%%%%%%%%%% Aufbau, Dateien %%%%%%%%%%%%%%%%%%%%%%%%%%%

\section{Introduction}\label{sec:intro}
In the mid seventies, Gromov introduced the bounded cohomology of a space 
and showed that  it vanishes in all degrees $n\geq1$ for simply connected CW-complexes \cite{Gromov_82}.
Brooks pointed out that this implies that the bounded cohomology 
of a space is isomorphic to the one of its fundamental group \cite{Brooks}.
In the latter note the author also made the first step towards the relative homological algebra
approach to the bounded cohomology of groups.   Ivanov then developed this
approach (with trivial coefficients) \cite{Ivanov}, 
incorporating the seminorm into the theory.  This led to the final form of Gromov's
theorem, namely that for a countable CW-complex the bounded cohomology 
is isometrically isomorphic to the bounded cohomology of its fundamental group.
We emphasize that, here and in the sequel, the coefficients are the trivial module $\RR$.

After Gromov's seminal paper \cite{Gromov_82}, bounded cohomology has admitted 
many generalizations and applications in a variety of contexts, though, ever since 
Ivanov's proof of Gromov's theorem, a sore point in the theory has been to establish 
that a given isomorphism is isometric. In this note, we will study this question 
for relative homology and bounded cohomology.

Bounded cohomology can be defined for pairs $(X,Y)$ of spaces, 
where $Y$ is a subspace of the space $X$, and there is an exact sequence
\bqn
\xymatrix@1{
\dots\ar[r]
&\hb^{n-1}(Y)\ar[r]^-{\delta^n}
&\hb^n(X,Y)\ar[r]^-{j^n}
&\hb^n(X)\ar[r]^{i^n}
&\hb^n(Y)\ar[r]
&\dots\,,
}
\eqn
where $j^n$ is induced by the inclusion of the corresponding cochain
complexes, $i^n$ is induced by the restriction map  
and $\delta^n$ is the connecting homomorphism.

A striking consequence of this long exact sequence can be obtained when we assume that 
each connected component of $Y$ has amenable fundamental group.
Indeed, as observed by Trauber in the 70's, one of the characteristic features
of bounded group cohomology  is that it vanishes for amenable groups in degree $n\geq1$.
This implies that $j^n$ is an isomorphism of vector spaces for $n\geq2$. 
In low degree, the isomorphism does not hold. Instead, it follows from $\hb^1(X)=0$ that we have an exact sequence
\bqn
\xymatrix@1{\hb^{0}(X,Y)\ \ar@{^{(}->}[r]^-{j^0}
&\hb^0(X)\ar[r]^-{i^0}
&\hb^0(Y)\ar@{->>}[r]^-{\delta^1}
&\hb^1(X,Y)\,.
}
\eqn
If $X$ is path connected then $\hb^0(X,Y)=0$ and $\hb^0(X)=\mathbb{R}$, 
while $\hb^0(Y)=\ell^\infty(\pi_0(Y))$. Here and in the sequel, $\ell^\infty(S)$ 
is the Banach space of all bounded real valued functions on the set $S$.

%A sore point in the theory has been the question whether, under the above hypotheses, $j^n$ is isometric.
%Our main goal is to prove the following:
%that a given isomorphism is isometric.
Our main result is that, under the above hypotheses,
$j^n$ is isometric:  

\begin{thm_intro}\label{thm:main}  Let $X\supseteq Y$ be a pair of countable CW-complexes.  
Assume that each connected component of $Y$ has amenable fundamental group.
Then the morphism obtained from the inclusion
\bqn
\xymatrix@1{
j^n:\hb^n(X,Y)\ar[r]
&\hb^n(X)
}
\eqn
is an isometric isomorphism for every $n\geq2$.
\end{thm_intro}

\subsection*{Gromov's Equivalence Theorem}
Beside the obvious $\ell^1$-seminorm, the relative homology of a pair of spaces can be endowed with a whole one parameter family of seminorms introduced by Gromov~\cite[Section 4.1]{Gromov_82}.  Indeed, let $(X,Y)$ be a pair of 
topological spaces and take a singular chain $c\in \ch_n(X)$ for $n\in \mathbb{N}$. Then, 
for every $\theta\geq 0$, one can define 
a norm on $\ch_n(X)$  by setting
\begin{equation*}\label{Gromov's norm}
\|c\|_1(\theta)=\|c\|_1+\theta\|d_n c\|_1\, . 
\end{equation*}
Taking the infimum value over the suitable sets of representatives,
this norm induces a seminorm on the relative homology module $\h_\bullet(X,Y)$, which is still
denoted by $\|\cdot \|_1(\theta)$. Notice that, for every $\theta\in [0,\infty)$, the norm $\|\cdot \|_1(\theta)$ is equivalent, but not equal, to the 
usual $\ell^1$-norm $\|\cdot\|_1=\|\cdot\|_1(0)$ on $\h_\bullet(X,Y)$.
%Notice that, for every $\theta\in [0,\infty)$, the norm $\|\cdot \|_1(\theta)$ is equivalent, but not equal, to the 
%usual $L^1$-norm $\|\cdot\|_1=\|\cdot\|_1(0)$. In particular,
%a relative cochain is bounded with respect to $\|\cdot \|_1(\theta)$
%if and only if it is bounded with respect to $\|\cdot \|_1$. 
%The dual norm of $\|\cdot \|_1(\theta)$ on relative cochains is denoted by $\|\cdot\|_\infty(\theta)$. It induces a seminorm on relative bounded cohomology,
%which is still denoted by $\|\cdot\|_\infty(\theta)$. 
By
passing to the limit, one can also define the seminorm
$\|\cdot\|_1(\infty)$, which however need not be equivalent 
to $\|\cdot\|_1$. For example, $\|\alpha\|_1(\infty)=\infty$ when $\alpha\in \h_n(X,Y)$ 
is such that $\|\partial_n \alpha\|_1>0$, where $\partial_n\colon \h_n(X,Y)\to \h_{n-1}(Y)$ is the connecting
homomorphism of the sequence of the pair. 

The following result is stated by Gromov in~\cite{Gromov_82}
(see also Remark~\ref{original} for a comment about Gromov's
original statement). However, Gromov's proof of Theorem~\ref{equi:teo} is not carried out in details and 
relies on the rather technical theory of multicomplexes. 
In Section~\ref{equi:sec} we provide a complete and direct proof of 
Theorem~\ref{equi:teo} as a consequence of Theorem~\ref{thm:main}.

\begin{thm_intro}[{Equivalence Theorem, \cite[page 57]{Gromov_82}}]\label{equi:teo}
Let $X\supseteq Y$ be a pair of countable CW-complexes.
If the fundamental groups of all connected components of $Y$ are amenable,
%then the norms $\|\cdot\|_\infty(\theta)$ on $H_b^n(X,Y)$, for $n\geq 2$,
%are equal for every $\theta\in[0,\infty]$, and
then %also 
the seminorms $\|\cdot\|_1(\theta)$ on $\h_n(X,Y)$, for $n\geq 2$,
are equal for every $\theta\in[0,\infty]$.
\end{thm_intro}

In~\cite{Park_hom}, Park uses a mapping cone construction to compute
the relative $\ell^1$-homology of topological pairs, and endows this $\ell^1$-homology
with a one parameter family of seminorms. This approach may also be exploited
in the case of singular homology, and in this context it is not difficult
to show that Park's seminorms coincide with Gromov's. 
A dual mapping cone construction is then used in~\cite{Park} 
to define a one parameter family of dual seminorms on relative bounded 
cohomology\footnote{Unless otherwise stated, we understand that relative bounded cohomology
is endowed with the seminorm introduced by Gromov in~\cite[9, Section 4.1]{Gromov_82}, 
which is induced by the $\ell^\infty$-norm on relative cochains (see also Section~\ref{sec: Resolutions}).
This is the case, for example, in the statement of Theorem~\ref{thm:main}.}. 
The arguments developed in Section~\ref{equi:sec} for the proof of Theorem~\ref{equi:teo}, 
which are inspired by Park's approach, may further be refined to prove that Park's seminorms on cohomology coincide with the
usual Gromov seminorm, provided that the map $\hb^\bullet(X,Y)\to \hb^\bullet (X)$ 
is an isometric isomorphism\footnote{In general, Park's seminorms are different from 
Gromov's \cite[Proposition 6.4]{Frigerio_Pagliantini_measure}.}. 
Together with Theorem~\ref{thm:main} and a standard duality argument, 
this fact can be exploited to provide another proof of Theorem~\ref{equi:teo}.

%In~\cite{Park}, Park
%makes use of a dual mapping cone construction in order to
%define a one parameter family of dual seminorms
%on relative bounded cohomology. 
%The arguments developed in Section~\ref{equi:sec}
%(which are inspired by Park's constructions)
%may be refined to prove that Park's seminorms on cohomology coincide with the
%usual Gromov seminorm, provided that the map
%$\hb^\bullet(X,Y)\to \hb^\bullet (X)$ is an isometric isomorphism. Together with Theorem~\ref{thm:main} and a standard duality
%argument, this fact can be exploited to provide another 
%proof of Theorem~\ref{equi:teo}.

%Building on Theorem~\ref{thm:main}  
%one can prove that Park's seminorms on cohomology coincide with Gromov's
%seminorm on $\hb^\bullet(X,Y)$, provided that $\h (j^n)\colon \hb^n(X,Y)\to \hb^n(X)$
%is an isometric isomorphism.
%Our proof of Theorem~\ref{equi:teo} is essentially based on a combination
%of these facts with a standard duality argument. 

As noticed by Gromov, Theorem~\ref{equi:teo} admits the following equivalent formulation, which is inspired 
by Thurston~\cite[Section 6.5]{Thurston_notes}:

\begin{thm_intro}\label{Thurston's version}
Let $X\supseteq Y$ be a pair of countable CW-complexes, and
suppose that the fundamental groups of all the components of $Y$ are amenable.
Let $\alpha\in \h_n(X,Y)$, $n\geq 2$. Then,
for every $\epsilon>0$ there exists a representative $c\in \ch_n(X)$ of $\alpha$ such that 
$\|c\|_1<\|\alpha\|_1+\epsilon$ and $\|d_n c\|_1<\epsilon$.
\end{thm_intro}

Theorem~\ref{Thurston's version} plays an important role
in several results about the 
(relative) simplicial volumes of glueings and fillings. In Section~\ref{Thurston} we provide
a proof of the equivalence between the statements of Theorem~\ref{equi:teo}
and~\ref{Thurston's version}.

Let $(X,Y)$ satisfy the hypothesis of Theorem~\ref{Thurston's version}, 
and let $n\geq 2$.
Via L\"oh's  "translation mechanism" \cite{Loeh}, 
Theorem~\ref{thm:main} implies that the natural map 
$\h_n^{\ell^1}(X)\to \h_n^{\ell^1}(X,Y)$  
on $\ell^1$-homology
is an isometric isomorphism. As a consequence, since the 
maps $\h_n(X)\to \h_n^{\ell^1}(X)$
and $\h_n(X,Y)\to \h_n^{\ell^1}(X,Y)$ induced by the inclusions of singular chains 
in $\ell^1$-chains are norm preserving~\cite{Loeh},
the homology map $j_n\colon\h_n(X)\rightarrow \h_n(X,Y)$ 
is norm preserving, although it is surely not an isomorphism in general. This implies
that every class lying in $j_n(\h_n(X))\subseteq \h_n(X,Y)$
satisfies the conclusion of Theorem~\ref{Thurston's version}.  
%Theorem \ref{Thurston's version} below is a non-trivial strengthening of that fact.

In the general case, the isometric isomorphism between 
$\h_n^{\ell^1}(X,Y)$ and $\h_n^{\ell^1}(X)$ ensures that
every ordinary relative homology
class $\alpha\in \h_n(X,Y)$ may be represented 
(in the corresponding relative $\ell^1$-homology module)
by an absolute $\ell^1$-cycle
$c$ whose norm is close to $\|\alpha\|_1$. One may wonder whether
the finite approximations $c_i$ of $c$ may be used to construct
the representative required in Theorem~\ref{Thurston's version}, since
the $\ell^1$-norm of $d_nc_i$ is approaching zero as $i$ tends to infinity. 
However, it is not clear how to control the support of $d_n c_i$, which may not be contained
in $Y$.

\subsection*{Simplicial volume of generalized Dehn fillings}
Let $M$ be the natural compactification of a complete finite volume hyperbolic 
$n$-manifold with toric cusps.
A \emph{generalized Dehn filling}
of $M$ was defined by Fujiwara and Manning in~\cite{Fujiwara_Manning_JDG} 
as the space obtained by replacing the cusps of the interior of $M$ with compact partial cones of their
boundaries (see Section~\ref{filling:sec} for a precise definition). Moreover,
in~\cite{Fujiwara_Manning} they proved that the simplicial volume 
does not increase under generalized Dehn filling. Note that in dimension $3$, 
the notion of generalized Dehn filling coincides with the usual notion of Dehn filling,
and the fact that the (relative) simplicial volume of any cusped hyperbolic $3$-manifold
strictly decreases under Dehn filling is a classical result by Thurston~\cite{Thurston_notes}. 

Fujiwara and Manning's argument
easily extends to the case in which  the
fundamental group of $M$ is residually finite and the inclusion of each boundary
torus in $M$ induces an injective map on fundamental groups. Recall that 
these conditions are always fulfilled if the interior of $M$
is a complete finite-volume hyperbolic manifold.
In Section~\ref{filling:sec} we generalize Fujiwara and Manning's result to the
case of an arbitrary manifold with toric boundary:

\begin{thm_intro}\label{Dehn filling}
Let $M$ be a compact orientable
$n$-manifold with boundary given by a union of tori, and let $N$ 
be a generalized Dehn filling of $M$.
Then
$$\|N \|\leq \|M,\partial M\|.$$
\end{thm_intro}

We provide two slightly different proofs of Theorem~\ref{Dehn filling}. 
The first one makes use of the Equivalence Theorem 
(or more precisely its reformulation given in Theorem~\ref{Thurston's version}), 
the other one relies directly  on Theorem~\ref{thm:main}.

\section{Resolutions in bounded cohomology}\label{sec: Resolutions}

Let $X$ be a space, where here and in the sequel by a space we will always mean a countable CW-complex. 
We denote by $\cb^n(X)$ the complex of bounded real valued
$n$-cochains on $X$ and, if $Y\subset X$ is a subspace, 
by $\cb^n(X,Y)$ the subcomplex of those bounded cochains
that vanish on simplices with image contained in $Y$.  All these spaces of cochains
are endowed with the $\ell^\infty$-norm and the corresponding cohomology groups
are equipped with the corresponding quotient seminorm.

For our purposes, it is important to  observe that the universal covering map
$p:\widetilde X\to X$ induces an isometric identification of the complex 
$\cb^n(X)$ with the complex $\cb^n(\widetilde X)^\Gamma$
of $\Gamma:=\pi_1(X)$-invariant bounded cochains on $\widetilde X$.
Similarly, if $Y':=p^{-1}(Y)$, we obtain an isometric identification
of the complex $\cb^n(X,Y)$ with the complex 
$\cb^n(\widetilde X,Y')^\Gamma$ of $\Gamma$-invariants
of $\cb^n(\widetilde X,Y')$.  

The main ingredient in the proof of Theorem~\ref{thm:main}, 
which is also essential in the proof of Gromov's theorem, 
is the result of Ivanov \cite{Ivanov} that the complex of  $\Gamma$-invariants of \bqn
\xymatrix{
\RR\ar[r]
&\cb^0(\widetilde X)\ar[r]
&\cb^1(\widetilde X)\ar[r]
&\dots
}
\eqn
computes the bounded cohomology of $\Gamma$. In fact, we will use the more precise statement 
that the latter complex is a strong resolution of $\RR$ by relatively injective $\Gamma$-Banach modules 
(see \cite{Ivanov} for the definitions of strong resolutions and relatively injective modules). 

By standard homological algebra techniques  \cite{Ivanov}, 
it follows from the fact that $\cb^n(\widetilde X)$ is a strong resolution by $\Gamma$-modules and 
$\ell^\infty(\Gamma^{\bullet+1})$ is a cochain complex (even a strong resolution) 
by relatively injective $\Gamma$-modules that there exists   a $\Gamma$-morphism of complexes
\bq \tag{$\Diamond
$}\label{eq: g_n}
\xymatrix@1{
g^n:\cb^n(\widetilde X)\ar[r]
&\ell^\infty(\Gamma^{n+1})
}
\eq
extending the identity, and such that $g^n$ is contracting, i.e. $\|g^n\|\leq1$, for $n\geq0$. 
This map induces Ivanov's isometric isomorphism $\hb^\bullet(X)\rightarrow \hb^\bullet(\Gamma)$.
%\begin{thm_intro}[\cite{Ivanov}]\label{thm:ivanov} The cochain complex
%is a {\red strong} {\green resolution} of $\RR$ by {\blue relatively injective} {\yellow $\Gamma=\pi_1(X)$-Banach modules}.
%\end{thm_intro}

The second result we need lies at the basis of the fact that the bounded cohomology of $\Gamma$ 
can be computed isometrically from the complex of bounded functions on any amenable $\Gamma$-space.
We will need only a particular case of the isomorphism, which is the existence of a contracting map 
between the complex $\ell^\infty(\Gamma^{n+1})$ 
and  the complex of alternating bounded functions $\ell_{\mathrm{alt}}^\infty(S^{n+1})$ 
when $S$ is a discrete amenable $\Gamma$-space. This is a very special case of \cite{Monod_book}, for which we present a direct proof. 

\begin{prop}[{\cite[Theorem~7.2.1]{Monod_book}}]\label{prop:monod}
Assume that $\Gamma$ is a group acting on a set $S$ such that 
all stabilizers are amenable subgroups of $\Gamma$.  Then for $n\geq0$ 
there is a $\Gamma$-morphism of complexes
\bqn
\xymatrix{
\mu^n:\ell^\infty(\Gamma^{n+1})\ar[r]
&\ell_{\mathrm{alt}}^\infty(S^{n+1})\, }
\eqn
extending $\id_\RR:\RR\to \RR$ that is contracting.%, i.e.  for all $n\geq0$,
%\bqn
%\|\mu_n\|\leq1\,.
%\eqn
\end{prop}

\begin{proof} Alternation gives a contracting $\Gamma$-morphism of complexes 
\bqn
\xymatrix{
\ell^\infty(S^{n+1})\ar[r]
&\ell^\infty_{\mathrm{alt}}(S^{n+1})\,,
}
\eqn
so that it suffices to construct a contracting $\Gamma$-morphism of complexes
\bqn
\xymatrix{
\mu^n:\ell^\infty(\Gamma^{n+1})\ar[r]
&\ell^\infty(S^{n+1})\,.
}
\eqn

We first construct $\mu^0$ and then inductively $\mu^n$, for $n\geq1$.  
Identify $S$ with a disjoint union $\sqcup_{i\in I}\Gamma/\Gamma_i$ 
of right cosets, where $\Gamma_i<\Gamma$ is amenable
and let $\lambda_i\in\ell^\infty(\Gamma_i)^\ast$ be a left $\Gamma_i$-invariant mean,
for every $i\in I$.
We define $\mu^0:\ell^\infty(\Gamma)\to\ell^\infty(S)$ for $f\in\ell^\infty(\Gamma)$
by setting $\mu^0(f)(\gamma\Gamma_i)$ to be the $\Gamma_i$-invariant mean $\lambda_i$ 
of the bounded function $\Gamma_i\to\RR$ defined by $\eta\mapsto f(\gamma\eta)$.
Clearly $\mu^0(\one_\Gamma)=\one_S$, so that $\mu^0$ extends 
$\id_\RR:\RR\to\RR$ and $\|\mu^0\|\leq1$.

Assume now that  we have defined $\mu^{n-1}:\ell^\infty(\Gamma^n)\to\ell^\infty(S^n)$.
Then we define $\mu^n$ as the composition of the following maps:
\bqn
\xymatrix{
\ell^\infty(\Gamma^{n+1})\ar[r]^-= \ar@{-->}[ddd]_{\mu^n}
&\ell^\infty(\Gamma\times\Gamma^n)\ar[r]^-\cong
&\ell^\infty(\Gamma,\ell^\infty(\Gamma^n))\ar[d]\\
&
&\ell^\infty(\Gamma,\ell^\infty(S^n))\ar[d]^-\cong\\
&
&\ell^\infty(S^n,\ell^\infty(\Gamma))\ar[d]\\
\ell^\infty(S^{n+1})
&\ell^\infty(S\times S^n)\ar[l]_-=
&\ell^\infty(S^n,\ell^\infty(S))\ar[l]\,,
}
\eqn
where $\cong$ denotes a Banach space isomorphism, 
while the first vertical arrow is induced by $\mu^{n-1}:\ell^\infty(\Gamma^n)\to\ell^\infty(S^n)$
and the third by $\mu^0:\ell^\infty(\Gamma)\to\ell^\infty(S)$.
Since all morphisms involved are contracting and 
equivariant for suitable $\Gamma$-actions, the same holds for $\mu^n$.
Finally one verifies that $(\mu^n)_{n\geq0}$ is a morphism of complexes.
\end{proof}

\section{Proof of Theorem~\ref{thm:main}}\label{sec:proof}

Let, as above, $p:\widetilde X\to X$ be the universal covering map,
$\Gamma:=\pi_1(X)$ and $Y=\sqcup_{i\in I}C_i$ the decomposition of $Y$
into a union of connected components. If  $\check C_i$ is a choice of a connected component of $p^{-1}(C_i)$ 
and $\Gamma_i$ denotes the stabilizer of $\check C_i$ in $\Gamma$ then
\bqn
p^{-1}(C_i)=\bigsqcup_{\gamma\in\Gamma/\Gamma_i}\gamma\check C_i\,.
\eqn

Let $\Ff\subset\widetilde X\smallsetminus Y'$ be a fundamental domain 
for the $\Gamma$-action on $\widetilde X\smallsetminus Y'$, where $Y'=p^{-1}(Y)$ as before. 
Define the $\Gamma$-equivariant map 
\bqn
r:\widetilde X\to S:=\Gamma\sqcup\bigsqcup_{i\in I}\Gamma/\Gamma_i
\eqn
as follows:
\bqn
r(\gamma x):=
\begin{cases}
\hphantom{\Gamma}\,\gamma\in\Gamma&\text{ if }x\in\Ff,\\
\gamma\Gamma_i\in \Gamma/\Gamma_i &\text{ if } x \in\check C_i\,.
\end{cases}
\eqn
For every $n\geq0$ define 
\bqn
\xymatrix@1{
r^n:\ell^\infty_{\mathrm{alt}}(S^{n+1})\ar[r]&\cb^n(\widetilde X)}
\eqn
by 
\bqn
r^n(c)(\sigma)=c(r(\sigma_0),\dots,r(\sigma_n))\,,
\eqn
where $c\in\ell^\infty_{\mathrm{alt}}(S^{n+1})$ and $\sigma_0,...,\sigma_n\in \widetilde{X}$ 
are the vertices of a singular simplex $\sigma:\Delta^n\to\widetilde X$.
Clearly $(r^n)_{n\geq0}$ is a $\Gamma$-morphism of complexes extending the identity on $\RR$
and $\|r^n\|\leq1$ for all $n\geq0$.

Observe that if $n\geq1$ and $\sigma(\Delta^n)\subset Y'$, 
then there are $i\in I$ and $\gamma\in\Gamma$ such that $\sigma(\Delta^n)\subset\gamma\check C_i$.
Thus
\bqn
r(\sigma_0)=\dots=r(\sigma_n)=\gamma\Gamma_i
\eqn
and thus
\bqn
r^n(c)(\sigma)=c(\gamma\Gamma_i,\dots,\gamma\Gamma_i)=0\,,
\eqn
since $c$ is alternating.
This implies that %, for these values of $n$, 
the image of $r^n$ is in $\cb^n(\widetilde X,Y')$. Thus we can write $r^n=j^n\circ r_1^n$,
where $j^n:\cb^n(\widetilde X,Y')\hookrightarrow\cb^n(\widetilde X)$ is the inclusion
and $r_1^n:\ell^\infty_{\mathrm{alt}}(S^{n+1})\to\cb^n(\widetilde X,Y')$ is a norm decreasing $\Gamma$-morphism
that induces a norm non-increasing map\footnote{To avoid confusion, henceforth
we use a different notation for the chain and cochain maps and the 
induced homology and cohomology maps. 
This is contrary to our notation in the introduction.} 
in cohomology 
\bqn
\xymatrix@1{\h(r_1^n):\h^n(\ell^\infty_{\mathrm{alt}}(S^{\bullet+1})^\Gamma)\ar[r]&\hb^n(X,Y)}\,,
\eqn
for $n\geq1$.

%Next, since $\Gamma_i$ is a quotient of $\pi_1(C_i)$ and hence amenable,
%Proposition~\ref{prop:monod} provides us with a $\Gamma$-morphism
%of complexes 
%\bqn
%\xymatrix@1
%{\mu_n:\ell^\infty(\Gamma^{n+1})\ar[r]&\ell^\infty_{\mathrm{alt}}(S^{n+1})}
%\eqn
%extending the identity and such that $\|\mu_n\|\leq1$ for $n\geq0$.
Using the map $g^n$ defined in (\ref{eq: g_n}) and the map $\mu^n$ provided by 
Proposition ~\ref{prop:monod} since, for all $i$, the group $\Gamma_i$ is a quotient of $\pi_1(C_i)$, and hence amenable,
we have the following diagram
\begin{equation*}
\xymatrix{\cb^n(\widetilde{X})  \ar@{-->}[rrrd]_{\mathrm{extends \ Id}_\mathbb{R}} \ar[r]^{g^n} 
& \ell^\infty(\Gamma^{n+1})  \ar[r]^{\mu^n} & \ell^\infty_{\mathrm{alt}}(S^{n+1}) \ar[rd]^{r^n}\ar[r]^{r_1^n}_{\mathrm{for \ }n\geq 1} 
&\cb^n(\widetilde{X},Y') \ar[d]^{j^n} \\
&&&\cb^n(\widetilde{X}),} 
\end{equation*}
where the dotted map is the composition $r^n\circ \mu^n \circ g^n$ 
which is a $\Gamma$-morphism  of strong resolutions by relatively injective modules
extending the identity, and hence induces the identity on 
$\hb^n(X)=\h^n(\cb^\bullet(\widetilde X)^\Gamma)$.

We proceed now to show that, for $n\geq 2$, the map
\bqn
\xymatrix@1{
\h(j^n):\hb^n(X,Y)\ar[r]&\hb^n(X)
}
\eqn
induced by $j^n$  is an isometric isomorphism in cohomology. 
In view of the long exact sequence for pairs in bounded cohomology and 
the fact that $\hb^\bullet(Y)=0$ in degree greater than $1$, we already know that $\h(j^n)$ is an  isomorphism. 
Let us denote by $\psi^n$ the map induced in cohomology by the composition
$r_1^n\circ \mu^n\circ g^n$. 
From the above it follows that
\bqn
\h(j^n)\circ \psi^n=\id_{\hb^n(X)}\,.
\eqn
Let $y\in\hb^n(X,Y)$ and set $x=\h(j^n)(y)$. 
%\bq\label{eq:x}
%x=\h(j_n)(y)\,.
%\eq
Then $\h(j^n)(\psi^n(x))=x$
and, as $\h(j^n)$ is injective, we get 
$
y=\psi^n(x)
$.
Since the maps $\h(j^n)$ and $\psi^n$ are norm non-increasing it follows that 
$$\|x\|_\infty=\ \|\h(j^n)(y)\|_\infty \leq \|y\|_\infty \mathrm{\ \ and \ \ }  \| y \| _\infty=\| \psi^n(x)\|_\infty \leq \| x\|_\infty$$ 
so that $\|\h(j^n)(y)\|_\infty=\|x\|_\infty=\|y\|_\infty$ and hence $\h(j^n)$ is norm preserving.

\section{Proof of Theorem~\ref{equi:teo}}\label{equi:sec}
Recall from the introduction that Gromov endowed 
the homology module $\h_n(X,Y)$ with a one parameter family of seminorms
$\|\cdot \|_1(\theta)$,  $\theta\in [0,\infty]$.
By definition, $\|\cdot \|_1(0)$ is equal to the usual $\ell^1$-seminorm
$\|\cdot\|_1$, while $\|\cdot \|_1(\infty)$ is defined by taking the limit of
$\|\cdot\|_1(\theta)$ as $\theta$ tends to infinity. Therefore, in order to prove Theorem~\ref{equi:teo}
it is sufficient to show that, if every component of $Y$ has amenable 
fundamental group, then $\|\cdot \|_1(\theta)=\|\cdot \|_1$ for every $\theta\in (0,\infty)$.

As is customary, % when dealing with seminorms in homology,
in order to compare seminorms in homology we will compare cocycles 
in bounded cohomology, and exploit the duality between homology and cohomology 
that is usually provided, in this context, by the Hahn-Banach Theorem (see \emph{e.g.}~\cite{Loeh}
for a detailed discussion of this issue).

In what follows, if $c\in \ch_n(X)$ is a representative of a class $\alpha\in \h_n(X,Y)$, 
with a slight abuse of notation $d_nc$ will be used to identify both
the element $d_n c\in \ch_{n-1}(X)$ and its preimage in $\ch_{n-1}(Y)$ via the inclusion
$i_{n-1}\colon \ch_{n-1}(Y)\to \ch_{n-1}(X)$. As mentioned in the introduction,
the proof of the following result is inspired by some techniques 
developed by Park in~\cite{Park_hom} and~\cite{Park}.

\begin{prop}\label{hiddencone:prop}
Let $\theta\in (0,\infty)$ and take $\alpha\in \h_n(X,Y)$.
% and take a representative $c\in \ch_n(X)$
%of $\alpha$. 
Then there exist $f\in \cb^n(X)$,
$g\in \cb^{n-1}(Y)$ such that the following conditions hold:
\begin{enumerate}
 \item 
$d^nf=0$ and
%\item 
$i^n(f)=-d^{n-1} g$;
\item
$f(c)+g(d_n c)=\|\alpha\|_1(\theta)$ for every representative $c\in \ch_n(X)$ of $\alpha$;
\item 
$\|f\|_\infty\leq 1$.
\end{enumerate}
\end{prop}
\begin{proof}
 Let us consider the direct sum
$$
V=\ch_n(X)\oplus \ch_{n-1} (Y)\, ,
$$
and endow $V$ with the norm $\|\cdot\|_1(\theta)$ defined by
$$
\|(u,v)\|_1(\theta)=\|u\|_1+\theta\|v\|_1\, .
$$
Let us also set 
%\begin{equation}\label{cochain cone}
$$
V^*=\cb^n(X)\oplus \cb^{n-1}(Y)\, ,
$$
%\end{equation}
and endow $V^*$ with the $\ell^\infty$-norm $\|\cdot\|_\infty (\theta)$ defined by
$$
\|(f,g)\|_\infty(\theta)=\max\{\|f\|_\infty ,\theta^{-1}
\|g\|_\infty\}\ .
$$
It is readily seen that $(V^*,\|\cdot\|_\infty(\theta))$ is isometrically identified to the topological dual of $(V,\|\cdot\|_1(\theta))$ via the pairing 
$$
V^*\times V
\to\mathbb{R},\qquad 
((f,g),(u,v))\mapsto f(u)+g(v)\, .
$$
%realizes $V^*$ as the  topological dual of $V$,
%and an easy computation shows that the norm $\| \cdot \|_\infty(\theta)$
%on $V^*$ coincides with the operator norm with respect
%to the norm $\|\cdot \|_1(\theta)$ on $V$.
Let $\bdh_n(X)\subseteq \ch_n(X)$ be the space of absolute $n$-boundaries of $X$, and let us set
$W_1=\bdh_n(X)\oplus \{0\}\subseteq V$. We also set $$
W_2=\{(u,v)\in V\, |\, u=i_{n} (z)\, , v=d_n z\ {\rm for\ some}\ z\in \ch_{n}(Y)\}\, \subseteq\, V\, ,
$$
and $W=W_1+W_2$.
It is easy to verify that two relative cycles $c$, $c'\in \ch_n(X)$ represent the same element in
$\h_n(X,Y)$ if and only if $(c,d_nc)-(c',d_nc')$ lies in $W$.
Let $c\in \ch_n(X)$ be any representative of $\alpha\in \h_n(X,Y)$. 
Our previous remark implies that
\begin{equation}\label{dist:eq}
\|\alpha\|_1(\theta)=\inf \{\|(c,d_nc)-w\|_1(\theta)\, |\, w\in W\}={\rm dist}((c,d_n c), W)\, ,
\end{equation}
where the last distance  is computed of course with respect to the norm $\|\cdot \|_1(\theta)$ on $V$. 

Now, an easy application of the Hahn-Banach Theorem ensures that we may find a functional $(f,g)\in V^*$
such that the following conditions hold:
\begin{enumerate}
\item[(a)] 
$0=(f,g)(u,v)=f(u)+g(v)$ for every $(u,v)\in W$;
 \item[(b)] 
 $(f,g)(c,d_n c)=f(c)+g(d_n c)={\rm dist}((c,d_n c), W)=\|\alpha\|_1(\theta)$;
\item[(c)] $\|(f,g)\|_\infty(\theta)=1$.
\end{enumerate}
The fact that $(f,g)$ vanishes on $W_1$ implies that $d^n f=0$, while $(f,g)|_{W_2}=0$ implies that $i^n(f)=-d^{n-1}g$, so (a) implies point 
(1) of the statement. Point (1) implies in turn that $f(c')+g(d_n c')=f(c)+g(d_n c)$ for every representative
$c'\in\ch_n(X)$ of $\alpha$, so
point (2) is a consequence of (b). Since $\|f\|_\infty\leq \|(f,g)\|_\infty(\theta)$, 
point (3) is a consequence of (c).
\end{proof}

We are now ready to prove Theorem~\ref{equi:teo}. Suppose that the fundamental group of every component of $Y$ is amenable, let $n\geq 2$
and take an element $\alpha\in \h_n(X,Y)$. Also take $\theta\in (0,\infty)$. 
Since the inequality $\|\alpha\|_1\leq \|\alpha\|_1(\theta)$ is obvious, we need to show that 
$$
\|\alpha\|_1(\theta)\leq \|\alpha\|_1\, .
$$

Let $f\in \cb^n(X)$, $g\in \cb^{n-1}(Y)$ be chosen as in the statement
of Proposition~\ref{hiddencone:prop}. Of course we may extend $g$ to an element $\hat{g}\in \cb^{n-1}(X)$
such that $i^{n-1}(\hat{g})=g$ and $\|\hat{g}\|_\infty=\| g\|_\infty$ (for example, we may extend $g$ to zero on simplices
that are not contained in $Y$). Let now $f'=f+d^{n-1} \hat{g}$,
and let $c\in \ch_n(X)$ be any representative of $\alpha$.
 By point (2) of Proposition~\ref{hiddencone:prop} we have
\begin{equation}\label{fprime}
f'(c)=(f+d^{n-1}\hat{g})(c)=f(c)+\hat{g}(d_n c)=\|\alpha\|_1(\theta)\, .
\end{equation}
Point (1)
of Proposition~\ref{hiddencone:prop} imply that $f'$ is a relative cocycle.
%Moreover, point (2) of Proposition~\ref{hiddencone:prop} implies that
%$$
%i^n (f')=i^n(f)+d^{n-1}(i^n(\hat{g}))=i^n(f)+d^{n-1}g=0\, ,
%$$
%so $f'$ is a relative cocycle, and 
We denote by $[f']\in\hb^n(X,Y)$ the corresponding relative cohomology class.

 Let us now recall that by Theorem~\ref{thm:main} the map
$$
\h(j^n)\colon \hb^n(X,Y)\to \hb^n(X)
$$ is an isometric isomorphism.
If $[f]$ denotes the class of $f$ in $\hb^n(X)$, then 
the equality
$f'=f+d^{n-1}\hat{g}$ implies that $\h(j^n)([f'])=[f]$.
As a consequence of Theorem~\ref{thm:main}, this gives in turn that
$$
\|[f']\|_\infty =\|[f]\|_\infty \leq 1\, ,
$$
where the last inequality follows from point (3) of Proposition~\ref{hiddencone:prop}.
In other words, for every $\varepsilon>0$ we may find a cochain $f''\in \cb^{n-1}(X,Y)$ such that $\|f'+d^{n-1}f''\|_\infty\leq 1+\varepsilon$.
Since $f''$ vanishes on $d_n c\in \ch_{n-1}(Y)$, using Equation~\eqref{fprime} we may now conclude that
$$
\|\alpha\|_1(\theta)=f'(c)=(f'+d^{n-1}f'')(c)\leq \|f'+d^{n-1}f''\|_\infty \|c\|_1\leq (1+\varepsilon)\|c\|_1\, .
$$
Since this inequality holds for every representative $c\in \ch_n(X)$ of $\alpha$ and for every $\varepsilon>0$, we finally have
that
$$
\|\alpha\|_1(\theta)\leq \|\alpha\|_1\, .
$$
This concludes the proof of Theorem~\ref{equi:teo}.

\begin{rem}\label{original}
Since the norm $\|\cdot \|_1(\theta)$ on relative
chains is equivalent to the usual
norm $\|\cdot\|_1=\|\cdot\|_1(0)$  for every $\theta\in [0,\infty)$, 
a relative cochain is bounded with respect to $\|\cdot \|_1(\theta)$
if and only if it is bounded with respect to $\|\cdot \|_1$. Therefore, one may endow
the set $\cb^\bullet(X,Y)$ with
the dual norm $\|\cdot\|_\infty(\theta)$ of $\|\cdot \|_1(\theta)$. 
This norm induces a seminorm on the relative bounded cohomology
module $\h^\bullet (X,Y)$,
already introduced by Gromov in~\cite{Gromov_82},
and still denoted by $\|\cdot\|_\infty(\theta)$. Gromov's original formulation
of the Equivalence Theorem contains two statements.
The second one (which is the most widely exploited for applications) is just Theorem~\ref{equi:teo} stated above. The first one
asserts that, when all the components of $Y$ have amenable fundamental groups, the seminorm  $\|\cdot\|_\infty(\theta)$ on $\hb^n(X,Y)$, $n\geq 2$,
does not depend on $\theta$. 
%In order to prove this statement it would be sufficient
%to check that the isomorphism
%$$
%\h(\beta^\bullet)\colon (\hb^\bullet(X,Y),\|\cdot\|_\infty(\theta))\to 
%(\hb^\bullet(Y\to X),\|\cdot\|_\infty(\theta))
%$$ 
%is isometric for every $\theta$. Moreover, a direct proof of this fact would probably
%allow us to avoid the use of mapping cones. It is interesting to notice
%that the analogous result
%in the context of homology (\emph{i.e.}~Lemma~\ref{lem: theta iso}) admits a very simple proof.
\end{rem}

\section{Thurston's seminorms on relative homology\\ and proof of Theorem~\ref{Thurston's version}}\label{Thurston}
Let us describe a family of seminorms on $\h_n(X,Y)$ that was introduced by
Thurston~\cite[Section 6.5]{Thurston_notes}.
For every $\alpha\in \h_n(X,Y)$ and $t>0$ we set
$$\|\alpha\|_{(t)}=\inf\{\|z\|_1\, |\, z\in \ch_n(X,Y),\, [z]=\alpha,\, \|d_n z\|_1\leq t\}.$$
Note that we understand that $\inf \emptyset=+\infty$. Following Thurston\footnote{The norm $\|\alpha\|_{(0)}$ is denoted by $\|\alpha\|_{0}$ in \cite{Thurston_notes}. We introduce the parenthesis in the notation for $\|\alpha\|_{(t)}$ to avoid any ambiguity when $t=1$.}, we set
$$\|\alpha\|_{(0)}=\lim_{t\rightarrow 0}\|\alpha\|_{(t)}=\sup_{t > 0}\|\alpha\|_{(t)}.$$

It readily follows from the definitions that Theorem~\ref{Thurston's version}
is equivalent to the statement that $\|\alpha\|_1=\|\alpha\|_{(0)}$ for every
$\alpha\in \h_n(X,Y)$, $n\geq 2$, provided that the fundamental group of each component
of $Y$ is amenable. Therefore, 
the equivalence between Theorems~\ref{equi:teo} and~\ref{Thurston's version}
is an immediate consequence of the following lemma
which is stated in~\cite[page 56]{Gromov_82} and proved here below
for the sake of completeness.

\begin{lemma}\label{GroThu}
 The seminorms $\|\cdot\|_1(\infty)$ and $\|\cdot \|_{(0)}$ on $\h_n(X,Y)$ coincide.
\end{lemma}
\begin{proof}
Take $\alpha\in \h_n(X,Y)$.
 For every $\theta\in[0,\infty)$ and $t>0$
we have
\begin{align*}
\|\alpha\|_1(\theta)&\leq \inf_z\{\|z\|_1(\theta)\, |\, [z]=\alpha,\, \|d_nz\|_1\leq t\}\\
&\leq \inf_z\{\|z\|_1+\theta t\, |\, [z]=\alpha,\, \|d_nz\|_1\leq t\}\\
&=\|\alpha\|_{(t)}+\theta t.
\end{align*}
By passing to the limit on the right side for $t\to 0$ we get
$\|\alpha\|_1(\theta)\leq \|\alpha\|_{(0)}$ for every $\theta\in[0,\infty)$, so
$\|\alpha\|_1(\infty)\leq \|\alpha\|_{(0)}$.

Let us now prove the other inequality.
Of course  we may restrict to the case $\|\alpha\|_1(\infty)< \infty$. Let 
us fix $\epsilon>0$. By definition there exists
a sequence $\{z_i\}_{i\in\mathbb{N}}\subseteq \ch_n(X)$ such that $[z_i]=\alpha$
and $$\|z_i\|_1+i\|d_n z_i\|_1\leq \|\alpha\|_1(i)+\epsilon\leq\|\alpha\|_1(\infty)+\epsilon.$$
Since $\|\alpha\|_1(\infty)< \infty$ 
the sequence $\{\|d_n z_i\|_1\}_i$ converges to $0$. As a consequence,
for every $\delta>0$
there exists  $i_0\in \mathbb{N}$ such that
 $\|d_n z_{i_0}\|_1\leq \delta$, so that
$$
\|\alpha\|_{(\delta)}\leq \|z_{i_0}\|_1\leq \|z_{i_0}\|_1+i_0\|d_n z_{i_0}\|_1\leq \|\alpha\|_1(\infty)+\epsilon\ .
$$
Since this estimate holds for every $\delta>0$, we may pass to the limit for $\delta\to 0$
and obtain the inequality $\|\alpha\|_{(0)}\leq \|\alpha\|_1(\infty)+\epsilon$,
whence the conclusion since $\epsilon$ is arbitrary. 
\end{proof}

\section{Simplicial volume of generalized Dehn fillings}\label{filling:sec}
Let us begin by recalling the definition of generalized Dehn filling~\cite{Fujiwara_Manning_JDG}.
Let $n\geq 3$ and let $M$ be a compact orientable
$n$-manifold such that 
$\partial M=N_1\cup\ldots\cup N_m$, where $N_i$ is an $(n-1)$-torus for every $i$.

For each $i\in\{1,\dots, m\}$ we put on $N_i$ a flat structure,
and we choose a totally geodesic $k_i$-dimensional
torus  $T_i\subseteq N_i$, where $1\leq k_i\leq n-2$.
Each $N_i$ is foliated by parallel copies of $T_i$ with leaf space
$L_i$ which is homeomorphic to a $(n-1-k_i)$-dimensional torus. 
The \emph{generalized Dehn filling} $M(T_1,\dots,T_m)$ is 
defined as the quotient of ${M}$
obtained  by collapsing
$N_i$ on $L_i$ for every $i\in\{1,\dots,m\}$. Observe that unless $k_i=1$ for every $i$, $M(T_1,\dots,T_m)$ is not a manifold. However, being a pseudomanifold,
$M(T_1,\dots,T_m)$
admits a fundamental class,
whence a well-defined simplicial volume.

%\begin{thm}[\cite{Fujiwara_Manning}]\label{FM}
%Let $M$ be as above,
%suppose that $V={\rm int} (M)$ admits a complete finite-volume hyperbolic structure,
%and let $M(T_1,\dots, T_m)$ be a filling of $M$.
%Then
%$$\|M(T_1,\dots,T_m)\|\leq \| M,\partial M\|=\frac{\Vol(V)}{v_n}\ .$$
%\end{thm}

%\begin{rem}
% Let us suppose that $V={\rm int} (M)$ supports a complete finite-volume hyperbolic structure. Then $M$ is homeomorphic
%to the manifold $V_0$ obtained by removing from $V$ a set of open disjoint horocusps. 
%Moreover, $\partial V_0$ admits a natural Euclidean structure, and if all the tori $T_i$ have injectivity radius strictly larger than $2\pi$ the pseudo-manifold
%$M(T_1,\dots,T_m)$ is called a $2\pi$-filling of $M$. 
%It is proved in \cite{Fujiwara_Manning} that 
%$\|M(T_1,\dots,T_m)\|>0$ for every  $2\pi$-filling $M(T_1,\dots, T_m)$ of $M$. 
%\end{rem}

%\subsection{Proof(s) of Theorem~\ref{Dehn filling}}
We propose two proofs of Theorem~\ref{Dehn filling}. The first proof follows very closely Fujiwara and Manning's
approach, which is in turn inspired by Thurston~\cite{Thurston_notes}.
In fact, in~\cite{Fujiwara_Manning} the authors provided both an explicit
homological proof of the Equivalence Theorem in the case
of manifolds with $\pi_1$-injective toric boundary and
residually finite fundamental group, and an explicit
proof of the uniform boundary condition for tori. The second alternative short proof of the theorem 
relies more directly on the isometry proved in Theorem~\ref{thm:main}, thus avoiding the explicit
use of the Equivalence Theorem.

\begin{proof}[First proof of Theorem~\ref{Dehn filling}] Let us set $N=M(T_1,\dots,T_m)$, $L=\sqcup_{i=1}^mL_i$ and let
$p_n\colon \ch_n(M,\partial M)\rightarrow \ch_n(N,L)$ be the map induced by the projection. 
By Theorem~\ref{Thurston's version}, for every $\epsilon>0$ there exists  a relative fundamental cycle 
$c\in \ch_n(M,\partial M)$
such that $\|c\|_1\leq\|M,\partial M\|+\epsilon$ and $\|d_n c\|_1\leq \epsilon$. 
Since each leaf space $L_i$ is  homeomorphic to a $(n-1-k_i)$-dimensional torus,
the cycle $p_n(d_nc)\in \ch_{n-1}(L)$ is a boundary.
Moreover, since $\pi_1(L)$ is amenable,
the module $\ch_{n-1}(L)$ satisfies Matsumoto-Morita's \emph{uniform boundary condition} \cite{Matsumoto_Morita}, so
there exist $K>0$ (independent of $c$) and $c'\in \ch_{n}(L)$ such that $d_nc'=p_n(d_nc)$ 
and $\|c'\|_1\leq K\|p_n(d_nc)\|_1$.
It is easy to check that  $p_n(c)-c'$ is a fundamental cycle for $N$ and
\begin{align*}
\|N\|&\leq\|p_n(c)-c'\|_1\leq\|p_n(c)\|_1+\|c'\|_1\\
&\leq \|c\|_1+\|c'\|_1\\
&\leq \|M,\partial M\|+\epsilon +K\|p_n(d_nc)\|_1\\
&\leq\|M,\partial M\|+\epsilon+K\epsilon.
\end{align*}
Since $\epsilon$ is arbitrary, this concludes the proof.
\end{proof}

%\begin{proof}[Alternative proof of Theorem~\ref{Dehn filling}]
\begin{proof}[Second proof of Theorem~\ref{Dehn filling}] 
Let $p:(M,\partial M)\rightarrow (N,L)$ 
be the projection map and $j\colon (N,\emptyset)\hookrightarrow (N,L)$ be the inclusion map.
By the exact sequence of the pair the inclusion map induces an isomorphism
$\h(j_n)\colon \h_n(N)\rightarrow \h_n(N,L)$.
Since every component of $L$ has amenable fundamental group,
Theorem~\ref{thm:main} implies that the inclusion map induces an isometric isomorphism 
$\h(j^n)\colon \hb^n(N,L)\rightarrow \hb^n(N)$ in bounded cohomology.
Via the translation principle, this implies in turn that also
$\h(j_n)\colon \h_n(N)\rightarrow \h_n(N,L)$ is an isometry.

Denoting by $\psi_n$ the inverse of $\h(j_n)$, it is easy to verify that  
$\psi_n(\h(p_n)([M,\partial M]))$ is a fundamental class for $N$.
%$$
% H_n(M,\partial M)\stackrel{H_n(p)}{\longrightarrow} H_n(N,V)
%\stackrel{\psi_n}{\longrightarrow}H_n(N).
%M,\partial M &\longmapsto & H_n(p)([M,\partial M]) & \longmapsto & [N].
%\end{array}
%$$
Since $\h(p_n)$ is contracting and 
$\psi_n$ is an isometry, we now have that
$$
\begin{array}{lllll}
\|N\|_1&=&\|\psi_n(\h(p_n)([M,\partial M]))\|_1&=&\|\h(p_n)([M,\partial M])\|_1\\&\leq&
\|[M,\partial M]\|_1&=&\|M,\partial M\|\ ,
\end{array}
$$
which finishes the proof of the theorem.
\end{proof}

\vskip1cm

\providecommand{\bysame}{\leavevmode\hbox to3em{\hrulefill}\thinspace}
\providecommand{\MR}{\relax\ifhmode\unskip\space\fi MR }
\providecommand{\MRhref}[2]{%
  \href{http://www.ams.org/mathscinet-getitem?mr=#1}{#2}
}
\providecommand{\href}[2]{#2}

%\newpage


\begin{thebibliography}{10}

\bibitem{Brooks}
Robert Brooks, \emph{Some remarks on bounded cohomology}, Riemann surfaces and
  related topics: {P}roceedings of the 1978 {S}tony {B}rook {C}onference
  ({S}tate {U}niv. {N}ew {Y}ork, {S}tony {B}rook, {N}.{Y}., 1978), Ann. of
  Math. Stud., vol.~97, Princeton Univ. Press, Princeton, N.J., 1981,
  pp.~53--63. \MR{624804 (83a:57038)}

%\bibitem{Bucher_Burger_Iozzi_mostow}
%Michelle Bucher, Marc Burger, and Alessandra Iozzi, \emph{A dual interpretation
%  of the {G}romov--{T}hurston proof of {M}ostow rigidity and volume rigidity
%  for representations of hyperbolic lattices}, Proceedings Conference "Trends
%  in Harmonic Analysis", Roma 2011, Springer Verlag, 2012.

%\bibitem{Burger_Iozzi_supq}
%Marc Burger and Alessandra Iozzi, \emph{Bounded {K}\"ahler class rigidity of
%  actions on {H}ermitian symmetric spaces}, Ann. Sci. \'Ecole Norm. Sup. (4)
%  \textbf{37} (2004), no.~1, 77--103. \MR{2050206 (2005b:32048)}

%\bibitem{Burger_Iozzi_Wienhard_toledo}
%Marc Burger, Alessandra Iozzi, and Anna Wienhard, \emph{Surface group
%  representations with maximal {T}oledo invariant}, Ann. of Math. (2)
%  \textbf{172} (2010), no.~1, 517--566. \MR{2680425}

%\bibitem{Clerc_Orsted}
%Jean-Louis Clerc and Bent {\O}rsted, \emph{The {G}romov norm of the {K}aehler
%  class and the {M}aslov index}, Asian J. Math. \textbf{7} (2003), no.~2,
%  269--295. \MR{2014967 (2005b:53127)}

%\bibitem{Domic_Toledo}
%Antun Domic and Domingo Toledo, \emph{The {G}romov norm of the {K}aehler class
%  of symmetric domains}, Math. Ann. \textbf{276} (1987), no.~3, 425--432.
%  \MR{875338 (88e:32057)}

\bibitem{Frigerio_Pagliantini_measure}
Roberto Frigerio and Cristina Pagliantini, \emph{Relative measure homology and
  continuous bounded cohomology of topological pairs},
  http://arxiv.org/abs/1105.4851, to appear in Pacific J.~Math.

\bibitem{Fujiwara_Manning_JDG}
Koji Fujiwara and Jason~F. Manning, \emph{${\rm CAT(0)}$ and ${\rm CAT}(-1)$ fillings of hyperbolic manifolds},
J. Differential Geom.  \textbf{85}  (2010),  no.~2, 229--269.
\MR{2732977 (2012d:57037)}

\bibitem{Fujiwara_Manning}
Koji Fujiwara and Jason~F. Manning, \emph{Simplicial volume and fillings of
  hyperbolic manifolds}, Algebr. Geom. Topol. \textbf{11} (2011), no.~4,
  2237--2264. \MR{2826938}

\bibitem{Gromov_82}
Michael Gromov, \emph{Volume and bounded cohomology}, Inst. Hautes \'Etudes
  Sci. Publ. Math. (1982), no.~56, 5--99 (1983). \MR{686042 (84h:53053)}

\bibitem{Ivanov}
Nikolay~V. Ivanov, \emph{Foundations of the theory of bounded cohomology}, Zap.
  Nauchn. Sem. Leningrad. Otdel. Mat. Inst. Steklov. (LOMI) \textbf{143}
  (1985), 69--109, 177--178, Studies in topology, V. \MR{806562 (87b:53070)}

\bibitem{Loeh}
Clara L\"oh, \emph{Isomorphisms in $l^1$-homology},
M\"unster Journal of Mathematics \textbf{1} (2008),  237--266. \MR{2502500 (2010b:55007)}


%\bibitem{Loh07}
%Clara L\"oh, \emph{$l^1$--homology and simplicial volume}, Ph.~D.~thesis,
%WWU M{\"u}nster (2007), available online at  http://nbn-resolving.de/urn:nbn:de:hbz:6-37549578216.

\bibitem{Matsumoto_Morita}
Shigenori Matsumoto and Shigeyuki Morita, \emph{Bounded cohomology of certain groups of homeomorphisms},  Proc. Amer. Math. Soc.  \textbf{94}  (1985),  no.~3, 539--544.
\MR{0787909 (87e:55006)}

\bibitem{Monod_book}
Nicolas Monod, \emph{Continuous bounded cohomology of locally compact groups},
  Lecture Notes in Mathematics, vol. 1758, Springer-Verlag, Berlin, 2001.
  \MR{1840942 (2002h:46121)}

\bibitem{Park_hom}
HeeSook Park, \emph{Foundations of the theory of $l_1$ homology},  
J. Korean Math. Soc.  \textbf{41}  (2004),  no.~4, 591--615.
\MR{2068142 (2005c:55011)}

\bibitem{Park}
HeeSook Park, \emph{Relative bounded cohomology}, Topology Appl. \textbf{131}
  (2003), no.~3, 203--234. \MR{1983079 (2004e:55008)}


\bibitem{Thurston_notes}
William P.~Thurston, \emph{Geometry and topology of 3-manifolds}, Notes from Princeton
  University, Princeton, NJ, 1978.

\end{thebibliography}
\end{document}